\documentclass[10pt,leqno, final]{amsart}
\usepackage[colorlinks,linkcolor=blue,citecolor=blue,urlcolor=blue]{hyperref}
\usepackage{amsfonts,amssymb}
\usepackage{color}
\usepackage{graphicx}
\usepackage[notref,notcite]{showkeys}
\usepackage{mathtools}
\usepackage{extarrows}
\usepackage{algorithmic}   
\usepackage{caption}
\usepackage{subcaption}
\usepackage{tikz}
\usepackage{mathrsfs}

\allowdisplaybreaks

\usepackage{todonotes}

\newtheorem{theorem}{Theorem}[section]
\newtheorem{lemma}[theorem]{Lemma}

\newtheorem{proposition}[theorem]{Proposition}

\theoremstyle{definition}

\newtheorem{example}[theorem]{Example}
\newtheorem*{example*}{Example}
\newtheorem{remark}[theorem]{Remark}

\numberwithin{equation}{section}

\usepackage{my_latex_commands}

\newcommand*{\Borel}{\mathfrak{B}}
\newcommand*{\Measure}{\mathfrak{M}}
\newcommand{\cbound}{\underbar{\textup{c}}}

\newcommand{\cost}{c_{\textup{d}}}
\newcommand{\marg}{\mu_2^{\textup{d}}}

\begin{document}

\title[Bilevel Kantorovich Problem, Part I]{Bilevel Optimization of the 
Kantorovich Problem and its Quadratic Regularization \\ Part I: Existence Results}

\thanks{This research was supported by the German Research Foundation (DFG) under grant number~LO 1436/9-1 within the priority program Non-smooth and Complementarity-based Distributed Parameter Systems: Simulation and Hierarchical Optimization (SPP~1962).}

\author{Sebastian Hillbrecht}
\address{Sebastian Hillbrecht, Technische Universit\"at Dortmund, Fakult\"at f\"ur Mathematik, Lehrstuhl X, Vogelpothsweg 87, 44227 Dortmund, Germany}
\email{sebastian.hillbrecht@tu-dortmund.de}

\author{Christian Meyer}
\address{Christian Meyer, Technische Universit\"at Dortmund, Fakult\"at f\"ur Mathematik, Lehrstuhl X, Vogelpothsweg 87, 44227 Dortmund, Germany}
\email{christian2.meyer@tu-dortmund.de}

\date{\today}

\begin{abstract} 
    This paper is concerned with an optimization problem governed by 
    the Kantorovich optimal transportation problem. This gives rise to a bilevel optimization problem, which can be reformulated as a mathematical problem with complementarity constraints in the space of regular Borel measures.
    Because of the non-smoothness induced by the complementarity relations, problems of this type are frequently regularized. Here we apply a quadratic regularization of the Kantorovich problem.
    As the title indicates, this is the first part in a series of three papers.
    It addresses the existence of optimal solutions to the bilevel Kantorovich problem and its quadratic regularization, whereas part II and III are dedicated to the convergence analysis for vanishing regularization.
\end{abstract}

\maketitle


\section{Introduction}

This paper is concerned with a bilevel optimization problem with the Kantorovich problem of optimal transport as the lower-level problem. The problem under consideration takes the following form:
\begin{equation}\tag{BK}\label{eq:BK}
    \left\{\quad
    \begin{aligned}
        \inf_{\pi,\mu_1} \quad & \JJ(\pi, \mu_1) \\
        \text{s.t.} \quad & \mu_1\in\Measure(\Omega_1), \quad \mu_1\geq 0,
        \quad \|\mu_1\|_{\Measure(\Omega_1)} = \|\marg\|_{\Measure(\Omega_2)}, \\
        & \pi \in \argmin\left\{\int_\Omega \cost\,\d\varphi
        \colon\varphi\in\Pi(\mu_1,\marg), \,\varphi\geq 0\right\}.
    \end{aligned}
    \right.        
\end{equation}
Herein, $\cost$ is a given cost function measuring the transportation cost and $\marg$ a given marginal on a domain $\Omega_2$. The set $\Pi(\mu_1, \marg)$ denotes the set of feasible transport plans, i.e., regular Borel measures that have $\mu_1$ and $\marg$ as first and second marginal, see \eqref{eq:Pi} below.
The lower level problem thus aims at minimizing the transportation cost among all feasible transport plans associated with $\mu_1$ and $\marg$.
It is Kantorovich's well-known generalization of the famous Monge problem, cf.~\cite{Kan42}.
We refer to \cite{Vil03, Vil09, AG13, San15} for more details on the Kantorovich problem and its application background.
The bilevel optimization problem now consists of varying the first marginal $\mu_1$ such that this marginal, together with an associated optimal transport plan, minimizes a given objective $\JJ$.
The additional constraints on $\mu_1$ in \eqref{eq:BK} ensure that there is at least one optimal transport plan associated with $\mu_1$ so that the feasible set of \eqref{eq:BK} is non-empty.
One possible application for such a bilevel problem could be, for instance, the identification of the first marginal based on measurements $\pi^{\textup{d}}$ of the transport plan on a part $D$ of the domain $\Omega_1 \times \Omega_2$.
In this case, the upper level objective would be of the form $\JJ(\pi, \mu_1) = d(\pi, \pi_{\textup{d}})$ (plus potential regularization terms accounting for errors in the measurement), where $d$ denotes a suitable distance such as $|\pi - \pi^{\textup{d}}|(D)$ or $\|\pi - \pi^{\textup(d)}\|_{W^{-1,p}(D)}$ for some $p > \dim(D)$.

From a bilevel optimization point of view, the Kantorovich problem is challenging.
First, for a given cost $\cost$ and marginals $\mu_1$ and $\marg$, the optimal transport plan needs not to be unique (unless the cost function is strictly convex and at least one of the marginals is absolutely continuous w.r.t.\ the Lebesgue measure, see \cite[Theorem~1.17]{San15}).
Thus, in general, there is no single-valued solution mapping $\mu_1 \mapsto \pi$ associated with the lower level problem in \eqref{eq:BK}.
This prevents us from using the so-called implicit programming approach, where the lower level problem is replaced by its solution operator and the (potentially limited) differentiability properties of the latter are used to derive optimality conditions and optimization algorithms for the bilevel problem, cf.\ e.g.\ \cite{Mig76, CCMW18}.
Alternatively, one could replace the convex lower-level problem by its necessary and sufficient first-order optimality conditions. These, however, contain a complementarity system in the space of regular Borel measures, which turns the bilevel problem into a mathematical program with complementarity constraints (MPCC) in $\Measure(\Omega_1\times \Omega_2)$. 

A common strategy to treat MPCCs is to regularize the complementarity constraints and the lower level problem. We only refer to \cite{LPR96, SS00, HKS13, KS15} in the finite dimensional setting and to \cite{Bar93, HK09, HMW12, SW13, Wac16} for problems in function spaces. These approaches are of theoretical as well as numerical interest.
While a limit analysis for vanishing regularization parameters yields stationarity conditions for the original problem, the regularized problems can often be treated with standard algorithms that, together with a path-following procedure for the regularization parameter, can provide an efficient method for solving an MPCC.
Here, we follow a similar approach and employ a quadratic regularization of the Kantorovich problem, which was proposed and analyzed in \cite{LMM21}.
This regularization has several advantages.
First, the regularized Kantorovich problem is strictly convex and thus uniquely solvable, which is in particularly attractive from the viewpoint of bilevel optimization as it allows the implicit programming approach to be applied.
Moreover, the regularity of the optimal transport plans is improved in a way that we are faced with an MPCC in $L^2(\Omega_1 \times \Omega_2)$ instead of an MPCC in the space of regular Borel measures.
Finally, as shown in \cite{LMM21}, the quadratic regularization preserves essential features of the original Kantorovich problem such as the sparsity of the optimal transport plan as well as a dual problem that provides a substantial reduction of the dimension.
As a price for these desirable properties, the regularized problems still contain a complementarity relation and are therefore not smooth, in contrast to common MPCC regularization approaches.
However, due to their particular structure involving a complementarity system in $L^2(\Omega_1\times \Omega_2)$, we expect that non-smooth optimization algorithms for MPCCs in function space are applicable, see e.g.\ \cite{HU19, CdlRM20}.

This work is the first part in a series of three papers. While the other two contributions address the question of convergence of solutions of the regularized bilevel problems for vanishing regularization parameter, this paper is concerned with the existence of solutions.
For the bilevel Kantorovich problem \eqref{eq:BK} itself, existence of globally optimal solutions is rather straightforward to show, based on known (weak) stability results for the Kantorivich problem.
The situation changes, if one turns to its regularized counterpart. As we will see by means of a counterexample, the solution operator associated with the regularized Kantorovich problem is not weakly continuous. Nevertheless, one can show its strong continuity, which allows us to prove the existence
of solution for the regularized bilevel problems.

The paper is organized as follows:
After introducing some basic notation and assumptions in the rest of this introduction, we collect some known results on the Kantorovich problem and its quadratic regularization in Section~\ref{sec:prelim}.
We then turn to the existence of globally optimal solutions for \eqref{eq:BK} 
in Section~\ref{sec:existBK}.
The main part of the paper is contained in Section~\ref{sec:existBKgam}, where we first verify that the regularized solution operator is locally H\"older continuous and, based on that, show the existence of optimal solutions for the regularized bilevel problems.

\subsection{Notation and Standing Assumptions}\label{sec:assu}
	
\subsubsection*{Domains}	
For $d_1,d_2\in\N$, let $\Omega_1\subset\R^{d_1}$ and $\Omega_2\subset\R^{d_2}$ be compact and connected sets with non-empty interior. We moreover suppose that their Cartesian product $\Omega\coloneqq\Omega_1\times\Omega_2$ coincides with the closure of its interior and has a Lipschitz boundary in the sense of \cite[Def.~1.2.2.1]{Gri85}.
By $\Borel(\Omega)$, we denote the respective Borel $\sigma$-algebra on $\Omega$ and by $\lambda$ the Lebesgue measure on $\Borel(\Omega)$.
For $\Omega_1$ and $\Omega_2$, $\Borel(\Omega_i)$ and $\lambda_i$, $i=1,2$, are defined analogously so that $\lambda = \lambda_1 \otimes \lambda_2$.
Furthermore, we abbreviate $|\Omega_1|\coloneqq\lambda_1(\Omega_1)$, $|\Omega_2|\coloneqq\lambda_2(\Omega_2)$, and $|\Omega|\coloneqq\lambda(\Omega)$.

\subsubsection*{Marginals}
Let $(X, \Borel(X))$ be a measurable space. Then, we denote by $\Measure(X)$ the space of (signed) regular Borel measures on $X$ equipped with the total variation norm, i.e., $\|\mu\|_{\Measure(X)} \coloneqq |\mu|(X)$.
If $\mu_1\in\Measure(\Omega_1)$ and $\mu_2\in\Measure(\Omega_2)$, then 
the set of transport plans between the marginals $\mu_1$ and $\mu_2$ is given by
\begin{equation}\label{eq:Pi}
    \Pi(\mu_1,\mu_2) \coloneqq \{\pi\in\Measure(\Omega)\colon {P_1}_\#\pi=\mu_1~\text{and}~{P_2}_\#\pi=\mu_2\},
\end{equation}
where, for $i=1,2$,
\begin{equation*}
    {P_i}_\#\pi\coloneqq\pi\circ P_i^{-1}\colon\Borel(\Omega_i)\to\R, 
\end{equation*}		
is the pushforward measure of $\pi$ via the projection $P_i\colon\Omega\ni(x_1,x_2)\mapsto x_i\in\Omega_i$. 
Note that $\Pi(\mu_1,\mu_2)=\emptyset$, if $\mu_1(\Omega_1)\neq\mu_2(\Omega_2)$.
Throughout the paper, $\marg \in \Measure(\Omega_2)$ is a fixed marginal satisfying 
$\marg \geq 0$ and, in order to ease notation, $\|\marg\|_{\Measure(\Omega_2)}  = 1$. 
The normalization condition is no restriction and can be ensured by re-scaling.

Given a measure space $(X, \AA, \mu)$, the Lebesgue space of $p$-th power absolutely integrable functions is denoted 
by $L^p(X, \mu)$, $p\in [1, \infty)$. If $X\subset \R^n$, $n\in \N$, is a Lebesgue measurable set and 
$\mu$ is the Lebesgue measure, we simply write $L^p(X)$.

\subsubsection*{Cost Function}
The cost function is assumed to satisfy $\cost \in W^{1,p}(\Omega)$, $p > d_1 + d_2$, 
where, with a slight abuse of notation, $W^{1,p}(\Omega)$ denotes the Sobolev space on $\interior(\Omega)$. 
Note that, due to the regularity of $\partial\Omega$, $W^{1,p}(\Omega)$ is compactly embedded in 
$C(\Omega)$, cf.\ e.g.\ \cite[Theorem 6.3]{AF03}. 
Thus, there exists a continuous representative of $\cost$, which we denote by the same symbol.

\subsubsection*{Bilevel Objective}
The functional $\JJ \colon \Measure(\Omega) \times \Measure(\Omega_1)\to\R$ 
is supposed to be lower semicontinuous w.r.t.\ weak-$\ast$ convergence.
Let us give an application-driven example for such a functional. Suppose that subsets 
$D\in \Borel(\Omega)$ and $D_1\in \Borel(\Omega_1)$ and 
measurements $\pi_\d \in \Measure(D)$ and $\mu_1^\d \in \Measure(D_1)$ are given. 
Then we set
\begin{equation*}
    \JJ(\pi, \mu_1) \coloneqq
    |\pi - \pi_\d|(D) + \nu\, |\mu_1 - \mu_1^\d|(D_1),
\end{equation*}
where $\nu \geq 0$ is a weighting parameter. The goal of the bilevel optimization is then
to adjust $\mu_1$ and $\pi$ such that the deviation between an optimal transport process 
and (possibly inaccurate) measurements thereof on subdomains becomes minimal.


\section{Preliminaries}\label{sec:prelim}

Given marginals $\mu_1 \in \Measure(\Omega_1)$, $\mu_2 \in \Measure(\Omega_2)$ and 
a measurable cost function $c\colon  \Omega \to [-\infty, \infty]$, 
the Kantorovich problem of optimal transport reads
\begin{equation}\tag{KP}\label{eq:KP}
    \left\{\quad
    \begin{aligned}
        \inf \quad & \KK(\pi) \coloneqq \int_\Omega c(x)\, \d \pi(x)\\
        \text{s.t.} \quad & \pi \in \Pi(\mu_1,\mu_2), \quad \pi \geq 0.
    \end{aligned}
    \quad \right.
\end{equation}

\begin{lemma}[{\cite[Theorem 4.1]{Vil09}}]\label{lem:kantexist}
    If $\mu_1, \mu_2 \geq 0$ and $\|\mu_1\|_{\Measure(\Omega_1)} = \|\mu_2\|_{\Measure(\Omega_2)}$
    and if $c$ is lower semicontinuous and bounded from below, then there exists an optimal solution of \eqref{eq:KP}.
\end{lemma}

Despite this existence result and its simple structure, the Kantorovich problem provides some challenging aspects, 
especially from a numerical perspective. First of all, its solution may be non-unique (although there are conditions 
which guarantee uniqueness, see e.g.\ \cite[Theorem~1.17]{San15}). More importantly, since $\pi$ ``lives'' on the Cartesian product 
of $\Omega_1$ and $\Omega_2$, the dimension of a discretized counterpart of \eqref{eq:KP} easily becomes 
so large that a numerical solution by means of standard LP-solvers is no longer possible.
Therefore, several penalization and relaxation methods have been proposed, that in combination with dualization, allow a significant reduction in the size of the problem.
The most popular method is probably the entropic regularization in combination with the well-known Sinkhorn algorithm, see e.g.\ \cite{Cut13, CLM21}. 

In this paper, we rely on a different strategy, namely the quadratic regularization that has been introduced in \cite{LMM21} and is as follows: Given a regularization parameter $\gamma > 0$, two marginals $\mu_1 \in L^2(\Omega_1)$, $\mu_2 \in L^2(\Omega_2)$, and a cost function $c\in L^2(\Omega)$, we consider
\begin{equation}\tag{$\text{KP}_\gamma$}\label{eq:KPgam}
    \left\{\quad
    \begin{aligned}
        \inf \quad & \KK_\gamma(\pi_\gamma) \coloneqq \int_\Omega c(x)\, \pi_\gamma(x)\, \d \lambda(x) 
        + \tfrac{\gamma}{2} \|\pi_\gamma\|_{L^2(\Omega)}^2 \\
        \text{s.t.} \quad & \pi_\gamma \in L^2(\Omega), \quad \pi_\gamma \geq 0 
        \quad\lambda\text{-a.e. in}~\Omega, \\
        & \int_{\Omega_2}\pi_\gamma(x_1, x_2) \,\d \lambda_2(x_2) 
        =\mu_1(x_1)\quad\lambda_1\text{-a.e in}~\Omega_1, \\
		& \int_{\Omega_1}\pi_\gamma(x_1, x_2) \,\d \lambda_1(x_1) 
		= \mu_2(x_2)\quad\lambda_2\text{-a.e in}~\Omega_2.
    \end{aligned}
    \right.
\end{equation}

\begin{lemma}[{\cite[Lemma~2.1, Theorem~2.11]{LMM21}}]\label{lem:quadreg} \ 

    \emph{(i)} Problem \eqref{eq:KPgam} admits a unique solution if and only if 
    $\mu_i \geq 0$ $\lambda_i$-a.e.\ in $\Omega_i$, $i=1,2$, and 
    $\|\mu_1\|_{L^1(\Omega_1)} = \|\mu_2\|_{L^1(\Omega_2)}$.

    \emph{(ii)}  If, in addition, there exist constants $\cbound > -\infty$ and $\delta > 0$ such that 
    $c \geq \cbound$ $\lambda$-a.e.\ in $\Omega$ and $\mu_i \geq \delta$ $\lambda_i$-a.e.\ in $\Omega_i$, 
    $i=1,2$, then $\pi_\gamma\in L^2(\Omega)$ is a solution of \eqref{eq:KPgam} if and only if 
    there exist functions $\alpha_1\in L^2(\Omega_1)$ and $\alpha_2\in L^2(\Omega_2)$ satisfying			
    \begin{subequations}\label{eq:KantProbL2OptSystem}
	\begin{alignat}{3}
	    \pi_\gamma - \frac{1}{\gamma} (\alpha_1 \oplus \alpha_2 - c)_+ & = 0 & \quad & \lambda\text{-a.e. in}~\Omega,
	    \label{eq:KantProbL2OptSystem-a} \\
        \int_{\Omega_2} \pi_\gamma(x_1, x_2) \,\d\lambda_2(x_2) &= \mu_1(x_1) & & \lambda_1\text{-a.e. in}~\Omega_1, 
        \label{eq:KantProbL2OptSystem-b}\\
        \int_{\Omega_1}\pi_\gamma(x_1, x_2) \,\d\lambda_1(x_1) &= \mu_2(x_2) & & \lambda_2\text{-a.e. in}~\Omega_1.
        \label{eq:KantProbL2OptSystem-c}
    \end{alignat}
	\end{subequations}
    Herein, $(\alpha_1 \oplus \alpha_2)(x_1,x_2)\coloneqq\alpha_1(x_1)+\alpha_2(x_2)$ $\lambda$-a.e.\ in $\Omega$
	refers to the direct sum of $\alpha_1 \in L^2(\Omega_1)$ and $\alpha_2\in L^2(\Omega_2)$, while, for given 
	$u\in L^2(\Omega)$, $(u)_+(x) \coloneqq \max\{u(x);0\}$ $\lambda$-a.e.\ in $\Omega$ denotes the pointwise 
	maximum. It is clear that both the direct sum and the pointwise maximum map $L^2(\Omega_1)\times L^2(\Omega_2)$ 
	and $L^2(\Omega)$, respectively, to $L^2(\Omega)$ so that \eqref{eq:KantProbL2OptSystem-a} is well defined.
	
	\emph{(iii)} Under the above assumptions, 
	the functions $\alpha_i\in L^2(\Omega_i)$, $i=1,2$, from (ii) solve the dual problem given by
	\begin{equation}
	    \left\{\quad 
	    \begin{aligned}
	        \max \quad & \Phi_\gamma(a_1, a_2) \coloneqq 
	        - \tfrac{1}{2} \|(a_1 \oplus a_2 - c)_+ \|_{L^2(\Omega)}^2 
	        + \gamma \sum_{i=1}^2\int_{\Omega_i} a_i \mu_i \, \d \lambda_i \\
	        \text{s.t.} \quad & a_i \in L^2(\Omega_i), \; i =1,2,
	    \end{aligned}
	    \right.
	\end{equation}
	and there is no duality gap, i.e., $\Phi_\gamma(\alpha_1, \alpha_2) = \KK_\gamma(\pi_\gamma)$.
\end{lemma}

The above results directly address the aforementioned challenges. Besides the uniqueness of the solution, the approach allows us to escape the curse of dimensionality. 
If one inserts \eqref{eq:KantProbL2OptSystem-a} into \eqref{eq:KantProbL2OptSystem-b} and 
\eqref{eq:KantProbL2OptSystem-c}, then a non-smooth system of equations for the 
dual variables $\alpha_1$ and $\alpha_2$ arises. We are then dealing with a problem in 
$L^2(\Omega_1) \times L^2(\Omega_2)$ instead of $L^2(\Omega_1 \times \Omega_2)$, which, 
after discretization, leads to a substantial reduction of the number of unknowns. 
Moreover, due to the $\max$-operator in \eqref{eq:KantProbL2OptSystem-a}, the sparsity pattern of the transport plans is better 
preserved compared to the entropic regularization. 
Finally, the structure of \eqref{eq:KantProbL2OptSystem} allows for the application of 
a semi-smooth Newton method, see \cite{LMM21} for more details.

The convergence of (sub-)sequences of solutions of \eqref{eq:KPgam} to 
solutions of \eqref{eq:KP} for $\gamma \searrow 0$ is addressed in \cite{ML21}.
To be more precise, it is shown that the objective of \eqref{eq:KPgam} $\Gamma$-converges 
to the objective of \eqref{eq:KP} w.r.t.\ weak-$\ast$ convergence in $\Measure(\Omega)$
as $\gamma \searrow 0$, provided that the original marginals in $\Measure(\Omega_i)$, $i=1,2$, 
are smoothed and the smoothing parameter is properly coupled with $\gamma$, 
see \cite[Theorem~4.2]{ML21}.

As outlined in the introduction, the goal of this and the companion papers is to employ the 
quadratic regularization for a bilevel optimization problem with the Kantorovich problem \eqref{eq:KP} as the constraint. 
The motivation for this approach is as follows: First, the uniqueness of solutions for \eqref{eq:KPgam} 
allows us to define a solution operator $\SS_\gamma (c, \mu_1, \mu_2) \mapsto \pi_\gamma$, 
which, in turn, enables us to employ the so-called implicit programming approach.
Secondly, we expect that $\SS_\gamma$ (or a discretization thereof) provides sufficient 
smoothness to use non-smooth optimization algorithms for the solution of the regularized bilevel problem.
Before we address the regularized bilevel problem, we turn to the optimization of the original problem \eqref{eq:KP} 
and show existence of at least one optimal solution to the latter in the upcoming section.


\section{Existence of Optimal Solutions of the Bilevel Kantorovich Problem}
\label{sec:existBK}

Let us first recall the bilevel optimization of the Kantorovich problem from the introduction:
\begin{equation}\tag{BK}
    \left\{\quad
    \begin{aligned}
        \inf_{\pi,\mu_1} \quad & \JJ(\pi, \mu_1) \\
        \text{s.t.} \quad & \mu_1\in\Measure(\Omega_1), \quad \mu_1\geq 0,
        \quad \|\mu_1\|_{\Measure(\Omega_1)} = 1, \\
        & \pi \in \argmin\left\{\int_\Omega \cost\,\d\varphi
        \colon\varphi\in\Pi(\mu_1,\marg), \,\varphi\geq 0\right\},
    \end{aligned}
    \right.        
\end{equation}
where $\cost \in W^{1,p}(\Omega)$, for $p>d_1 + d_2$, and $\marg \in \Measure(\Omega_2)$ 
denote a fixed cost functional and a fixed marginal, respectively, and $\JJ$ is a given objective, 
see Section~\ref{sec:assu}. To shorten notation, 
given $c\in C(\Omega)$ and $\mu_i \in \Measure(\Omega_i)$, $i=1,2$, we abbreviate the set 
of associated optimal transport plans by
\begin{equation}\label{eq:solsetKP}
    \SS(c, \mu_1, \mu_2) \coloneqq
    \argmin\left\{\int_\Omega c\,\d \varphi \colon\varphi\in\Pi(\mu_1,\mu_2), \,\varphi\geq 0\right\}.
\end{equation}
The essential tool to establish the existence of solutions to \eqref{eq:BK} is the following stability result 
for the Kantorovich problem. Its proof is based on the concept of $c$-cyclic monotonicity. For details, we refer to \cite[Section~5]{Vil09}.

\begin{lemma}[{Stability of the transport plan, \cite[Theorem~5.20]{Vil09}}]\label{lem:kantstabil}
    Let $c \in C(\Omega)$ be given and assume that $\{\mu_1^k\}_{k\in \N} \subset \Measure(\Omega_1)$ 
    and $\{\mu_2^k\}_{k\in \N} \subset \Measure(\Omega_2)$ are sequences that 
    satisfy $\mu_1^k, \mu_2^k\geq 0$ and 
    $\|\mu_1^k\|_{\Measure(\Omega_1)} = \|\mu_2^k\|_{\Measure(\Omega_2)}$ for all $k\in \N$
    and converge weakly-$\ast$ to $\mu_1\in \Measure(\Omega_1)$ and $\mu_2\in \Measure(\Omega_2)$. 
    Let moreover $\{\pi_k\}_{k\in \N}$ be a sequence of optimal transport plans associated with 
    $(\mu_1^k, \mu_2^k)$, i.e., $\pi_k \in \SS(c, \mu_1^k, \mu_2^k)$.
    Then there is a subsequence that converges weakly-$\ast$ to an optimal  transport plan
    $\pi \in \SS(c, \mu_1, \mu_2)$.
\end{lemma}

The above lemma is just a special case of \cite[Theorem~5.20]{Vil09}, 
where the cost function need not to be fixed. We underline that the notion of ``weak convergence'' in \cite{Vil09} (sometimes also called narrow convergence) coincides with weak-$\ast$ convergence
in our case, since $\Omega_1$, $\Omega_2$, and $\Omega$ are compact.

\begin{theorem}\label{thm:bikantexist}
    There exists at least one globally optimal solution to \eqref{eq:BK}.
\end{theorem}

\begin{proof}
    Based on Lemma~\ref{lem:kantstabil}, the result easily follows from the direct method of the calculus of variations. 
    
    First, thanks to Lemma~\ref{lem:kantexist}, the feasible set of \eqref{eq:BK}, denoted by $\FF$, is nonempty. 
    Thus, there exists a minimizing sequence $\{(\pi_k, \mu_1^k)\}_{k\in \N}$ so that 
    \begin{equation*}
        \lim_{k\to\infty} \JJ(\pi_k, \mu_1^k) = \inf_{(\pi, \mu_1) \in \FF}  \JJ(\pi, \mu_1)
        \in \R \cup \{-\infty\}.
    \end{equation*}
    The feasibility of the minimizing sequence implies
    \begin{equation*}
        \|\pi_k\|_{\Measure(\Omega)} 
        = \pi_k(\Omega_1 \times \Omega_2) 
        = ({P_1}_\# \pi_k)(\Omega_1) = \mu_1^k(\Omega_1) = \|\mu_1^k\|_{\Measure(\Omega_1)} = 1
        \quad \forall\, k \in \N
    \end{equation*}
    so that there is a subsequence, denoted by the same symbol to ease notation, that converges weakly-$\ast$ 
    to a limit $(\bar\pi, \bar\mu_1) \in \Measure(\Omega) \times \Measure(\Omega_1)$.
    Thus $(\mu_1^k, \marg) \weak^* (\bar \mu_1, \marg)$ in $\Measure(\Omega_1) \times \Measure(\Omega_2)$ 
    and consequently, Lemma~\ref{lem:kantstabil} implies $\bar\pi \in \SS(\cost, \bar\mu_1, \marg)$.
    In view of the constraints of the Kantorovich problem, 
    this also gives $\bar\mu_1 \geq 0$ and $\|\bar\mu_1\|_{\Measure(\Omega_1)} = 1$ and hence,
    $(\bar \pi, \bar\mu_1) \in \FF$, i.e., the weak limit is feasible.
    The optimality of the weak limit follows from the presupposed weak-$\ast$ lower semicontinuity of $\JJ$.
\end{proof}

\begin{remark}
    Lemma~\ref{lem:kantstabil} only requires the continuity of the cost function. For the mere existence result, 
    one could therefore relax the regularity assumption on the cost function to $\cost \in C(\Omega)$.
    The improved regularity of $\cost$ is however required for the analysis of the regularized bilevel problem, 
    and for this reason, we impose it as standing assumption.
    Moreover, Lemma~\ref{lem:kantstabil} also holds in Polish spaces and not only in compact sets. 
    One can therefore generalize the existence result of Theorem~\ref{thm:bikantexist}
    to a much broader class of domains $\Omega_1$ and $\Omega_2$.
\end{remark}

If one replaces \eqref{eq:KP} by its necessary and sufficient optimality conditions, then 
\eqref{eq:BK} can equivalently be rewritten as 
\begin{equation*}
    \eqref{eq:BK} \quad \Longleftrightarrow\quad \left\{\;\;
    \begin{aligned}
        \inf \quad & \JJ(\pi, \mu_1) \\
        \text{s.t.} \quad & \mu_1\in\Measure(\Omega_1), \quad \mu_1\geq 0,
        \quad \|\mu_1\|_{\Measure(\Omega_1)} = 1, \\
        & \pi \in \Measure(\Omega_1 \times \Omega_2), \quad \varphi \in C(\Omega_1), 
        \quad \psi \in C(\Omega_2),\\
        & {P_1}_\# \pi = \mu_1, \quad {P_2}_\#\pi = \marg, \\
        & \pi \geq 0,\quad 
        \varphi(x_1) + \psi(x_2) \leq \cost(x_1, x_2) \;\, \forall \, (x_1, x_2) \in \Omega_1\times \Omega_2,\\
        & \int_{\Omega_1\times \Omega_2} (\varphi \oplus \psi -\cost) \,\d\pi = 0.
    \end{aligned}
    \right.        
\end{equation*}    
We observe that the last two lines in the above reformulation of \eqref{eq:BK} form a 
complementarity system in $\Measure(\Omega_1\times \Omega_2)$ and $C(\Omega_1\times \Omega_2)$, so that \eqref{eq:BK} becomes an MPCC in the space of regular Borel measures,
as already mentioned in the introduction.
Even though several results for MPCCs are known, in particular when the cone defining the complementarity constraints is polyhedric, which is the case here, see \cite[Example~4.12]{Wac19}, problems of this type are typically smoothed or regularized, and we will do just that in the next section.


\section{Existence of Optimal Solutions of the Regularized Bilevel Problem}
\label{sec:existBKgam}

The regularized Kantorovich problem \eqref{eq:KPgam} clearly admits a solution only if the marginals are functions in $L^2(\Omega_1)$ and $L^2(\Omega_2)$.
Therefore, one needs to regularize the marginals, if the Kantorovich problem 
in \eqref{eq:BK} is replaced by \eqref{eq:KPgam}. But even if the marginals were functions in 
$L^2(\Omega_1)$ and $L^2(\Omega_2)$, one needs to smooth them considering the lack of (weak) continuity of the solution mapping associated with \eqref{eq:KPgam}, see Example~\ref{ex:counterhoelder} below.
What is more, in order to guarantee the existence of 
the dual variables $\alpha_1$ and $\alpha_2$ from Lemma~\ref{lem:quadreg}, 
the marginals need to be strictly positive, see \cite[Assumption~1]{LMM21}.
We therefore introduce the convolution \& constant shifting operators
\begin{equation}\label{eq:defTT}
    \TT_i^\delta \colon \Measure(\Omega_i) \ni \mu_i \mapsto 
    \varphi_i^\delta \ast \mu_i + \frac{\delta}{|\Omega_i^\delta|} \in L^2(\Omega_i^\delta), 
    \quad i = 1,2,
\end{equation}
which turn the marginals into smooth and strictly positive functions on $\Omega_1^\delta$ and $\Omega_2^\delta$.
Herein, $\delta > 0$ is a smoothing parameter, 
$\varphi_i^\delta \in C_c^\infty(\R^{d_i})$ denotes a standard mollifier with 
$\|\varphi_i^\delta\|_{L^1(\R^{d_i})} = 1$ and support in 
$\overline{B(0, \delta)} \subset \R^{d_i}$ and $\Omega_i^\delta \coloneqq \Omega_i + \overline{B(0, \delta)}$, 
$i=1,2$. Moreover, we set $\Omega_\delta \coloneqq \Omega_1^\delta \times \Omega_2^\delta$.
When mollifying $\mu_i$, we of course extend it by zero, i.e.,
\begin{equation*}
    (\varphi_i^\delta \ast \mu_i)(x) = \int_{\Omega_i} \varphi_i^\delta(x - y) \,\d\mu_i(y), 
    \quad x\in \Omega_i^\delta.
\end{equation*}
As a consequence, $\varphi_i^\delta \ast \mu_i \weak^\ast \mu_i$ in $\Measure(\Omega_i)$ 
as $\delta \searrow 0$, provided that the support of $\mu_i$ has positive distance to $\partial\Omega_i$,
which will be useful for the convergence analysis in the companion paper \cite{HMM22}, where the smoothing parameter
$\delta$ will be polynomially coupled with the regularization parameter $\gamma$.

Owing to Lemma~\ref{lem:quadreg} there exists a unique solution $\pi_\gamma \in L^2(\Omega_\delta)$ to 
\eqref{eq:KPgam} for costs in $L^2(\Omega_\delta)$ and marginals in $L^2(\Omega_i^\delta)$, $i=1,2$, that satisfy the conditions in Lemma~\ref{lem:quadreg}(i). We denote the associated solution operator by
\begin{equation*}
\begin{aligned}
    & \SS_\gamma \colon  L^2(\Omega_\delta) \times\MM_0 \ni (c, \mu_1, \mu_2)
    \mapsto \pi_\gamma \in L^2(\Omega_\delta), \\
    & \MM_0(\Omega_\delta) 
    \coloneqq \big\{ (\mu_1, \mu_2) \in L^2(\Omega_1^\delta) \times L^2(\Omega_2^\delta) \colon 
    \begin{aligned}[t]
        & \|\mu_1\|_{L^1(\Omega_1^\delta)} = \|\mu_2\|_{L^1(\Omega_2^\delta)}, \\[-0.5ex]
        & \; \mu_i \geq 0\text{ $\lambda_i$-a.e.\ in } \Omega_i^\delta, i=1,2 \big\}.
    \end{aligned}        
\end{aligned}
\end{equation*}
To ease notation, we suppress the dependency of $\pi_\gamma$ and $S_\gamma$ on $\delta$.
Furthermore, we introduce the extension-by-zero operator $\EE_\delta \colon C(\Omega) \to L^2(\Omega_\delta)$, 
whose adjoint $\EE_\delta^* \colon L^2(\Omega_\delta) \to \Measure(\Omega)$ is the associated restriction operator.
Now, we have everything at hand to formulate the regularized bilevel problem:
\begin{equation}\tag{BK$_\gamma^\delta$}\label{eq:BKgam}
    \left\{\quad
    \begin{aligned}
        \inf_{\pi_\gamma,\mu_1, c} \quad & \JJ_\gamma(\pi_\gamma, \mu_1, c) 
        \coloneqq \JJ(\pi_\gamma, \mu_1) + \tfrac{1}{p \gamma}\, \|c - \cost\|_{W^{1,p}(\Omega)}^p \\
        \text{s.t.} \quad & c\in W^{1,p}(\Omega), \quad
        \mu_1\in\Measure(\Omega_1), \quad \mu_1\geq 0,
        \quad \|\mu_1\|_{\Measure(\Omega_1)} = 1, \\
        & \pi_\gamma 
        = \EE_\delta^* \,\SS_\gamma\big( \EE_\delta\,c, \TT_1^\delta(\mu_1), \TT_2^\delta(\marg) \big).
    \end{aligned}
    \right.        
\end{equation}
As we will see in the proof of Theorem~\ref{thm:BKgamexist}, there holds 
$(\TT_1^\delta(\mu_1), \TT_2^\delta(\marg)) \in \MM_0(\Omega_\delta)$ 
such that $\pi_\gamma$ is well defined, cf. Lemma \ref{lem:quadreg}(i).
Compared to \eqref{eq:BK}, we not only replace the Kantorovich problem as the lower-level problem with its regularized counterpart, but also add the cost function $c$ to the set of optimization variables.
This is motivated by the so-called reverse approximation property, which is essential to show the convergence 
of minimizers of \eqref{eq:BKgam} towards solutions of the original unregularized bilevel problem, see the companion paper \cite{HM22}, where this property is elaborated for the finite dimensional counterparts of \eqref{eq:BK} and \eqref{eq:BKgam}.
This property requires a set of optimization variables that is sufficiently rich, as is also required, e.g., in the optimization of perfect plasticity, see \cite{MW21}. For this reason, $c$ is treated as an additional 
optimization variable.
After all, the penalty term in the upper-level objective $\JJ_\gamma$ will ensure that, in the limit, $c$ equals the given 
cost function $\cost$, see \cite{HM22}.

\begin{remark}
    Instead of regularizing w.r.t. the Lebesgue measure, one could also apply a regularization 
    w.r.t.\ the product measure of the marginals $\mu_1\otimes \mu_2$, i.e., 
    \begin{equation}\tag{$\widetilde{\text{KP}}_\gamma$}
        \left\{\quad
        \begin{aligned}
            \inf \quad & \int_\Omega c\, \d \pi + \tfrac{\gamma}{2} \int_\Omega \pi^2 \,\d (\mu_1 \otimes \mu_2)\\
            \text{s.t.} \quad & \pi \in L^2(\Omega_1\times \Omega_2, \mu_1\otimes \mu_2) \cap \Pi(\mu_1, \mu_2),
        \end{aligned}
        \right.
    \end{equation}
    where, with a slight abuse of notation, we use the same symbol for the Borel measure $\pi$ and its 
    density w.r.t.\ the product measure. Note that the constraint $\pi \in \Pi(\mu_1, \mu_2)$ does not imply that 
    $\pi$ is automatically absolutely continuous w.r.t.\ the product measure, as the counterexample 
    $\Omega_i = [0,1]$, $\mu_i=\lambda$, $i=1,2$, and $\pi = (\id, \id)_\# \lambda$ shows. 
    Hence, an additional regularization is also necessary in this case. Nevertheless, 
    a regularization w.r.t.\ to the product measure has several advantages. For instance, the marginals need not to be smoothed and the positivity assumption on $\mu_1$ and $\mu_2$ in \cite[Assumption~1]{LMM21} becomes superfluous. However, in the bilevel context, this approach does not seem to be promising: $\mu_1$ is an upper-level variable and therefore the space for $\pi$ is no longer fixed but depends on the optimization variable.
\end{remark}

The rest of this paper is dedicated to the existence of solutions to \eqref{eq:BKgam}. In this context, the continuity properties of $\SS_\gamma$ are of course essential and will be discussed in the following.

\subsection{H\"older Continuity of the Regularized Solution Operator}\label{sec:Hoelder}

The key tool in the existence proof for the unregularized bilevel problem in Theorem~\ref{thm:bikantexist} 
has been the stability of the Kantorovich problem w.r.t.\ (weak-$\ast$) perturbations of the marginals from 
Lemma~\ref{lem:kantstabil}.
Unfortunately, such a weak continuity result does not hold in case of the regularized Kantorovich problem, as we will demonstrate below by means of a counterexample.

\begin{example}\label{ex:counterhoelder}
    Let $\Omega_1 = \Omega_2 = [0,1]$, $\gamma = 1$, and $c(x_1, x_2) \coloneqq \frac{1}{4} |x_1 - x_2|^2$. 
    Moreover, define $f\colon \R \to \R$, $f(x) \coloneqq \sgn(\sin(2\pi x))$ and, for $n\in \N$, set
    \begin{equation*}
        \alpha_1^n(x_1) \coloneqq
        \begin{cases}
            f(n\, x_1) + \frac{9}{4}, & 0 \leq x_1 \leq \frac{1}{2}, \\
            f(n\, x_1) + \frac{5}{4}, & \frac{1}{2}  < x_1 \leq 1,
        \end{cases}
        \qquad
        \alpha_2^n(x_2) \coloneqq
        \begin{cases}
           0, & 0 \leq x_2 \leq \frac{1}{2}, \\
           -\frac{1}{2}, & \frac{1}{2}  < x_2 \leq 1.
        \end{cases}         
    \end{equation*}
    In view of \eqref{eq:KantProbL2OptSystem-a}, we further set 
    \begin{equation*}
    \begin{aligned}
        \pi_n(x_1, x_2) &\coloneqq 
        \frac1\gamma \big( \alpha_1^n(x_1) + \alpha_2^n(x_2) - c(x_1, x_2)\big)_+ \\
        &= \begin{cases}
            f(n\, x_1) + \frac{9}{4} - \frac{1}{4} |x_1 - x_2|^2, & 0 \leq x_1, x_2 \leq \frac{1}{2}, \\
            f(n\, x_1) + \frac{5}{4} - \frac{1}{4} |x_1 - x_2|^2, & \frac{1}{2} < x_1 \leq 1,\; 0 \leq x_2 \leq \frac{1}{2}, \\
            f(n\, x_1) + \frac{7}{4} - \frac{1}{4} |x_1 - x_2|^2, & 0 \leq x_1 \leq \frac{1}{2},\; \frac{1}{2} < x_2 \leq 1, \\
            \big(f(n\, x_1) + \frac{3}{4} - \frac{1}{4} |x_1 - x_2|^2\big)_+, & 
            \frac{1}{2} < x_1, x_2 \leq 1, 
        \end{cases}
    \end{aligned}
    \end{equation*}
    whose weak limit (w.r.t.\ weak convergence in $L^2(\Omega)$) is 
    \begin{equation*}
    \begin{aligned}
        \pi(x_1, x_2) = \begin{cases}
            \frac{9}{4} - \frac{1}{4} |x_1 - x_2|^2, & 0 \leq x_1, x_2 \leq \frac{1}{2}, \\
            \frac{5}{4} - \frac{1}{4} |x_1 - x_2|^2, & \frac{1}{2} < x_1 \leq 1,\; 0 \leq x_2 \leq \frac{1}{2}, \\
            \frac{7}{4} - \frac{1}{4} |x_1 - x_2|^2, & 0 \leq x_1 \leq \frac{1}{2},\; \frac{1}{2} < x_2 \leq 1, \\
            \frac{7}{8} - \frac{1}{8} |x_1 - x_2|^2, & \frac{1}{2} < x_1, x_2 \leq 1.
        \end{cases}
    \end{aligned}
    \end{equation*}
    Because the system in \eqref{eq:KantProbL2OptSystem} is necessary and sufficient for optimality, $\pi_n$ is the solution of \eqref{eq:KPgam} with the cost function defined above and marginals given by
    \begin{equation*}
        \mu_1^n(x_1) \coloneqq \int_0^1 \pi_n(x_1, x_2)\,\d x_2, \quad 
        \mu_2^n(x_2) \coloneqq \int_0^1 \pi_n(x_1, x_2)\,\d x_1. 
    \end{equation*}
    Note that, for all $n\in \N$, $\mu_i^n\geq {}^1\!/\!{}_{16}$ $\lambda_i$-a.e., $i=1,2$, such that 
    Lemma~\ref{lem:quadreg}(ii) is indeed applicable.
    Clearly, the weak convergence of $\pi_n$ implies that $\mu_i^n$ converges weakly 
    to $\mu_i \in L^2(\Omega_i)$, $i=1,2$, and the pointwise bound carries over to the weak limit
    such that Lemma~\ref{lem:quadreg}(ii) also holds for the limits $\mu_1$ and $\mu_2$.
    Accordingly, if the weak limit $\pi$ were a solution to the regularized 
    Kantorovich problem, then there would be dual solutions $\alpha_1 \in L^2(\Omega_1)$ and $\alpha_2 \in L^2(\Omega_2)$ so that
    \begin{equation*}
        \pi(x_1, x_2) = \big( \alpha_1(x_1) + \alpha_2(x_2) - c(x_1, x_2)\big)_+ 
        \quad \lambda\text{-a.e.\ in } [0,1]^2.
    \end{equation*}
    Because of $\pi > 0$, this is equivalent to 
    \begin{equation}\label{eq:alpha1plusalpha2bsp}
        \alpha_1(x_1) + \alpha_2(x_2) 
        = \begin{cases}
            \frac{9}{4}, & 0 \leq x_1, x_2 \leq \frac{1}{2}, \\
            \frac{5}{4}, & \frac{1}{2} < x_1 \leq 1,\; 0 \leq x_2 \leq \frac{1}{2}, \\
            \frac{7}{4} , & 0 \leq x_1 \leq \frac{1}{2},\; \frac{1}{2} < x_2 \leq 1, \\
            \frac{7}{8} + \frac{1}{8} |x_1 - x_2|^2, & \frac{1}{2} < x_1, x_2 \leq 1,
        \end{cases}
    \end{equation}
    $\lambda$-a.e.\ in $[0,1]^2$. This, however, leads to a contradiction: picking arbitrary Lebesgue points $\hat x_2 \in \bigl(0, \frac12\bigr)$ and $\tilde x_2 \in \bigl(\frac12, 1\bigr)$ of $\alpha_2$, \eqref{eq:alpha1plusalpha2bsp} implies
    \begin{equation*}
        \begin{rcases}
            \tfrac{9}{4} - \alpha_2(\hat x_2), & 0 \leq x_1 \leq \tfrac{1}{2} \\
            \tfrac{5}{4} - \alpha_2(\hat x_2), & \tfrac{1}{2} < x_1 \leq 1
        \end{rcases}
        = \alpha(x_1) = 
        \begin{cases}
            \tfrac{7}{4} - \alpha_2(\tilde x_2), & 0 \leq x_1 \leq \tfrac{1}{2} \\
            \tfrac{7}{8} + \tfrac{1}{8} |x_1 - \tilde x_2|^2 - \alpha_2(\tilde x_2), &  \tfrac{1}{2} < x_1 \leq 1
        \end{cases}
    \end{equation*}
    $\lambda_1$-a.e.\ in $[0,1]$. While the conditions in $\bigl[0, \frac12\bigr]$ imply that $\alpha_2(\hat x_2) - \alpha_2(\tilde x_2) = \frac{1}{2}$, 
    it must hold that $\alpha_2(\hat x_2) - \alpha_2(\tilde x_2) \leq \tfrac{3}{8} < \tfrac{1}{2}$ 
    on $\bigl(\frac12, 1\bigr]$, which yields the desired contradiction.
    Therefore, the weak limit $\pi$ cannot be the solution of \eqref{eq:KPgam} associated with the 
    limits $\mu_1$ and $\mu_2$, giving in turn that $\SS_\gamma$ is not weakly continuous.
\end{example}

We compensate the lack of weak continuity by means of the convolution operators $\TT_i^\delta$, $i=1,2$, which turn weakly convergent into strongly convergent sequences.
Still, we need to prove that $\SS_\gamma$ is continuous w.r.t.\ the strong topology of $L^2(\Omega_\delta)$. 
This is what we will show next. To this end, let $\gamma>0$ and $\delta>0$ be arbitrary, 
but fixed throughout this section.
To simplify the notation, we slightly chance the notation in comparison to the beginning of this section and
the consecutive subsection. First, we suppress the sub- and superscripts $\gamma$ and $\delta$ 
in the rest of this section, except for $\SS_\gamma$ and $\pi_\gamma$ in order to 
underline the difference to the solution of the unregularized problem \eqref{eq:KP} and its solution set in 
\eqref{eq:solsetKP}. Moreover, given $(\mu_1, \mu_2)\in \MM_0(\Omega)$, and $c\in L^2(\Omega)$, we set 
$\pi_\gamma \coloneqq  \SS_\gamma(c, \mu_1, \mu_2) \in L^2(\Omega)$ 
(i.e., without the restriction operator $\EE_\delta^*$).

We start with several auxiliary boundedness results. For this purpose, let us consider 
arbitrary but fixed data $c\in L^2(\Omega)$, $\mu_i\in L^2(\Omega_i)$, $i=1,2$, with 
\begin{equation}\label{eq:normal}
    \|\mu_1\|_{L^1(\Omega_1)} = \|\mu_1\|_{L^1(\Omega_2)} \eqqcolon m.
\end{equation}
Moreover, we suppose
that the assumptions of the second part of Lemma~\ref{lem:quadreg} are satisfied, i.e., there are constants 
$\cbound > -\infty$ and $\delta > 0$ such that 
\begin{equation}\label{eq:databound}
    c \geq \cbound \quad \lambda \text{-a.e.\ in } \Omega
    \quad \text{and} \quad
    \mu_i \geq \delta \quad \lambda_i \text{-a.e.\ in } \Omega_i, \;i=1,2.
\end{equation}
Note that the latter condition is automatically fulfilled in context of the bilevel problem 
\eqref{eq:BKgam} owing to the definitions of the operators $\TT_i^\delta$.
Moreover, the penalty term in the upper level objective ensures that the cost $c$ is 
bounded in $W^{1,p}(\Omega) \embed C(\Omega)$, which ensures the other condition 
thanks to the compactness of $\Omega$.

\begin{lemma}\label{lem:piest}
    There exists a constant $C = C(\gamma, m)>0$ such that 
    \begin{equation*}
        \|\pi_\gamma\|_{L^2(\Omega)} \leq C \big( \|\mu_1\|_{L^2(\Omega_1)} \,\|\mu_2\|_{L^2(\Omega_2)} 
        + \|c\|_{L^2(\Omega)}\big).
    \end{equation*}
\end{lemma}

\begin{proof}
    According to Lemma~\ref{lem:quadreg}, there exist $\alpha_1\in L^2(\Omega_1)$ and 
    $\alpha_2\in L^2(\Omega_2)$ such that the system in \eqref{eq:KantProbL2OptSystem} is fulfilled. 
    Multiplying \eqref{eq:KantProbL2OptSystem-b} and \eqref{eq:KantProbL2OptSystem-c} 
    with $\alpha_1$ and $\alpha_2$, respectively, integrating and adding the resulting equations leads to
    \begin{equation}\label{eq:SoluL2NormEq}
	    \gamma \|\pi\|_{L^2(\Omega)}^2 
	    = \int_{\Omega_1}\mu_1\alpha_1\, \d \lambda_1 + 
	     \int_{\Omega_2}\mu_2 \alpha_2 \, \d\lambda_2 - \int_\Omega\pi c \,\d\lambda,
    \end{equation}
    where we also used \eqref{eq:KantProbL2OptSystem-a}.
    Exploiting the normalization condition in \eqref{eq:normal},	we obtain
    \begin{equation*}
        \int_{\Omega_1}\mu_1\alpha_1\,\d \lambda_1 
        = \frac{1}{m}\int_{\Omega_1}\mu_1\alpha_1\int_{\Omega_2}\mu_2 \,\d\lambda_2\,\d\lambda_1
        = \frac{1}{m}\int_\Omega\left(\mu_1\otimes\mu_2\right)\alpha_1\d\lambda
    \end{equation*}
    and an analogous equation for $\mu_2\alpha_2$. Herein, 
    $(\mu_1\otimes\mu_2)\coloneqq\mu_1(x_1)\mu_2(x_2)$ for $\lambda$-a.a.\ $(x_1,x_2)\in\Omega$ 
    refers to the direct product of $\mu_1$ and $\mu_2$. One easily verifies that 
    $\mu_1 \otimes \mu_2 \in L^2(\Omega)$ with $\|\mu_1\otimes \mu_2\|_{L^2(\Omega)} =
    \|\mu_1\|_{L^2(\Omega_1)} \|\mu_2\|_{L^2(\Omega_2)}$. 
    Therefore, \eqref{eq:SoluL2NormEq} allows us to estimate
    \begin{align*}
        \gamma\|\pi\|_{L^2(\Omega)}^2 
        &\leq \frac{1}{m}\int_\Omega (\mu_1\otimes\mu_2)(\alpha_1 \oplus \alpha_2-c)_+\d \lambda
        + \frac{1}{m}\int_\Omega(\mu_1\otimes\mu_2)c \,\d\lambda - \int_\Omega\pi c\,\d\lambda \\
        &\leq \gamma\,m^{-1}\,\|\pi\|_{L^2(\Omega)} \|\mu_1\|_{L^2(\Omega_1)} \|\mu_2\|_{L^2(\Omega_2)} \\
        & \quad + m^{-1}\,\|\mu_1\|_{L^2(\Omega_1)} \|\mu_2\|_{L^2(\Omega_2)} \|c\|_{L^2(\Omega)} + 
        \|c\|_{L^2(\Omega)}\|\pi\|_{L^2(\Omega)}.
    \end{align*}
    and Young's inequality then gives the result.
\end{proof}

The next two lemmas address the boundedness of the dual variables $\alpha_1$ and $\alpha_2$ 
from Lemma~\ref{lem:quadreg}. We first observe that the dual variables cannot be bounded by the data in any norm, 
since they are not unique: if $(\alpha_1, \alpha_2)$ satisfies \eqref{eq:KantProbL2OptSystem} for $\pi_\gamma$,
then, for any $r\in \R$, $(\alpha_1 + r, \alpha_2 - r)$ does so, too.
In the rest of the paper, we therefore pick a particular dual solution, namely the one, where $r$ is chosen such that 
\begin{equation}\label{eq:vanishmean}
    \int_{\Omega_2} \alpha_2 \, \d \lambda_2 = 0.
\end{equation}
We call a dual solution $(\alpha_1, \alpha_2)$ satisfying \eqref{eq:vanishmean} \emph{zero-mean dual solution}. 
The above considerations show that, under the assumptions of Lemma~\ref{lem:quadreg}, there always 
exists a zero-mean dual solution. 
We further underline that even the zero-mean dual solution is in general not unique as the counterexample in \cite[Section~2.3]{LMM21} illustrates. Still, we can show that every zero-mean dual solution is bounded by the data $c$, $\mu_1$, and $\mu_2$. 
We start with an $L^1$ bound in the following

\begin{lemma}\label{lem:L1bound}
    Every zero-mean dual solution satisfies 
    \begin{equation*}
        \|\alpha_1\|_{L^1(\Omega_1)} + \|\alpha_2\|_{L^1(\Omega_2)} \\
        \leq C \big( \|\mu_1\|_{L^2(\Omega_1)} \|\mu_2\|_{L^2(\Omega_2)}  + \|c\|_{L^2(\Omega)} +1 \big)^2   
    \end{equation*}
    with a constant $C = C(\gamma, m, \delta) > 0$.
\end{lemma}

\begin{proof}
    Lemma~\ref{lem:quadreg}(iii) yields 
    \begin{equation*}
    \begin{aligned}
        \|c\|_{L^2(\Omega)} \, \|\pi_\gamma\|_{L^2(\Omega)} 
        & \geq - \KK_\gamma(\pi_\gamma) \\
        & = - \Phi_\gamma(\alpha_1, \alpha_2) \\
        & \geq - \frac1m \int_\Omega (\mu_1 \otimes \mu_2)(\alpha_1 \oplus \alpha_2 - c) \,\d \lambda
        - \frac1m \int_\Omega (\mu_1 \otimes \mu_2) c \, \d \lambda,
    \end{aligned}
    \end{equation*}
    which, in combination with \eqref{eq:KantProbL2OptSystem-a}, leads to 
    \begin{equation*}
    \begin{aligned}
        & \|c\|_{L^2(\Omega)} \big( \|\pi_\gamma\|_{L^2(\Omega)} + \|\mu_1 \otimes \mu_2\|_{L^2(\Omega)}\big) \\
        & \quad \geq  - \frac1m\int_\Omega (\mu_1 \otimes \mu_2)(\alpha_1 \oplus \alpha_2 - c)_+ \,\d \lambda
        + \frac1m \int_\Omega (\mu_1 \otimes \mu_2)(\alpha_1 \oplus \alpha_2 - c)_- \,\d \lambda \\
        & \quad \geq  - \frac{\gamma}{m} \int_\Omega (\mu_1 \otimes \mu_2) \pi_\gamma \,\d \lambda 
        + \frac{\delta^2}{m} \int_\Omega (\alpha_1 \oplus \alpha_2 - c)_- \,\d \lambda,
    \end{aligned}
    \end{equation*}
    where $\delta > 0$ is the constant from \eqref{eq:databound}. 
    Thus, thanks to Lemma~\ref{lem:piest}, we arrive at 
    \begin{equation*}
    \begin{aligned}
        \|\alpha_1 \oplus \alpha_2 \|_{L^1(\Omega)} 
        & \leq \|(\alpha_1 \oplus \alpha_2 - c)_+ \|_{L^1(\Omega)} 
        + \|(\alpha_1 \oplus \alpha_2 - c)_- \|_{L^1(\Omega)} 
        + \|c\|_{L^1(\Omega)} \\
        &  \leq 
        \gamma \|\pi_\gamma\|_{L^1(\Omega)} + \|c\|_{L^1(\Omega)} \\
         & \quad + \tfrac{m}{\delta^2} \big(
        \begin{aligned}[t]
            & \|c\|_{L^2(\Omega)} \|\pi_\gamma\|_{L^2(\Omega)} \\
            & + \tfrac{1}{m}\|c\|_{L^2(\Omega)} \|\mu_1 \otimes \mu_2\|_{L^2(\Omega)}
            + \tfrac{\gamma}{m} \|\mu_1 \otimes \mu_2\|_{L^2(\Omega)} \|\pi_\gamma\|_{L^2(\Omega)}\big)            
        \end{aligned} \\
        & \leq C \big( 
        \|\mu_1\|_{L^2(\Omega_1)} \|\mu_2\|_{L^2(\Omega_2)}  + \|c\|_{L^2(\Omega)} +1 \big)^2 
    \end{aligned}
    \end{equation*}
    with a constant $C>0$ depending on $\gamma$, $m$,  and $\delta$.
    To deduce a bound for $\alpha_1$ and $\alpha_2$ individually from this, we employ \eqref{eq:vanishmean} 
    to obtain
    \begin{equation*}
    \begin{aligned}
        \|\alpha_1 \oplus \alpha_2 \|_{L^1(\Omega)} 
        & = \sup_{\substack{\phi\in L^\infty(\Omega), \\ \|\phi\|_{\infty}\leq1}} 
        \int_\Omega\phi (\alpha_1 \oplus \alpha_2) \,\d \lambda \\
        &\geq \sup_{\substack{\phi_1\in L^\infty(\Omega_1), \\ \|\phi_1\|_\infty \leq1}} 
        \int_\Omega(\phi_1\otimes 1)(\alpha_1 \oplus \alpha_2) \,\d \lambda 
        = |\Omega_2| \, \|\alpha_1\|_{L^1(\Omega_1)}.
    \end{aligned}
    \end{equation*}
    Finally, the $L^1$-norm of $\alpha_2$ is estimated by 
    \begin{equation*}
        \|\alpha_2\|_{L^1(\Omega_2)} 
        \leq |\Omega_1|^{-1} \big(\|\alpha_1 \oplus \alpha_2\|_{L^1(\Omega)} 
        + |\Omega_2| \|\alpha_1\|_{L^1(\Omega_1)}\big),
    \end{equation*}
    which, together with the previous estimates, yields the claim.
\end{proof}

\begin{lemma}\label{lem:L2bound} 
    There is a constant $C = C(\gamma, m, \delta, \cbound) > 0$ such that 
    \begin{equation*}
         \|\alpha_1\|_{L^2(\Omega_1)} + \|\alpha_2\|_{L^2(\Omega_2)} \\
        \leq C \big( \|\mu_1\|_{L^2(\Omega_1)} \|\mu_2\|_{L^2(\Omega_2)}  + \|c\|_{L^2(\Omega)} +1 \big)^{6}
   \end{equation*}
    holds for every zero-mean dual solution.
\end{lemma}

\begin{proof}
    We proceed in two steps: in a first step, 
    we show the boundedness of the positive parts of $\alpha_1$ and $\alpha_2$ (i); 
    in a second step, we derive bounds for their negative parts (ii).
		
    Ad~(i): Let us denote  the bound from Lemma \ref{lem:L1bound} by 
    $M > 0$, i.e., in particular $\|\alpha_2\|_{L^1(\Omega_2)} \leq M$, 
    and define, up to sets of zero Lebesgue measure, the subset
    \begin{align*}
	    \tilde{\Omega}_2\coloneqq
	    \left\{x_2\in\Omega_2\colon|\alpha_2(x_2)|\leq\frac{2M}{|\Omega_2|}\right\} \subset\Omega_2.
    \end{align*}
    Then it follows by construction that
    \begin{equation*}
        M \geq
        \int_{\Omega_2}|\alpha_2(x_2)|\d \lambda_2
        \geq \int_{\Omega_2\setminus\tilde{\Omega}_2}|\alpha_2(x_2)|\d\lambda_2
        \geq\frac{2M}{|\Omega_2|} \, |\Omega_2\setminus\tilde{\Omega}_2|,
    \end{equation*}
    which in turn implies
    \begin{equation}\label{eq:OmegaTildeMass}
        \frac{|\Omega_2|}{2}\geq |\Omega_2\setminus\tilde{\Omega}_2| = |\Omega_2|-|\tilde{\Omega}_2|
        \quad \Longrightarrow \quad 
        |\tilde{\Omega}_2|\geq\frac{|\Omega_2|}{2}>0.
    \end{equation}
    Furthermore, we define, up to sets of zero Lebesgue measure, the sets
    \begin{equation*}
        \Omega_1^+ \coloneqq \{x_1\in\Omega_1\colon\alpha_1(x_1)\geq0\}
    \end{equation*}
    as well as
    \begin{equation*}
        \tilde\Omega^+ \coloneqq \{ (x_1, x_2) \in \Omega_1^+ \times \tilde\Omega_2 \colon
        \alpha_1(x_1) + \alpha_2(x_2) \geq 0 \}.
    \end{equation*}
    Then, by construction, it holds
    \begin{equation*}
        0 \leq \alpha_1(x_1) \leq  - \alpha_2(x_2) \leq \frac{2M}{|\Omega_2|}
        \quad \lambda\text{-a.e.-in } (\Omega_1^+ \times \tilde\Omega_2)\setminus \tilde\Omega^+.
    \end{equation*}
    Together with \eqref{eq:OmegaTildeMass}, this yields
    \begin{equation}\label{eq:a1plusest}
    \begin{aligned}
        \|(\alpha_1)_+\|_{L^2(\Omega_1)}^2
        &= \frac{1}{|\tilde\Omega_2|} \int_{\tilde\Omega_2} 
        \int_{\Omega_1^+} |\alpha_1|^2 \,\d\lambda_1\d\lambda_2 \\
        &\leq \frac{2}{|\Omega_2|} \Bigg( \int_{\tilde\Omega^+} |\alpha_1|^2 \, \d\lambda
        + \int_{(\Omega_1^+ \times \tilde\Omega_2)\setminus \tilde\Omega^+} |\alpha_1|^2 \, \d\lambda\Bigg)\\
        &\leq \frac{2}{|\Omega_2|}  \int_{\tilde\Omega^+} |\alpha_1|^2 \, \d\lambda
        + \frac{8\,|\Omega_1|}{|\Omega_2|^2} \,M^2.
    \end{aligned}
    \end{equation}
    In order to estimate the first term on the right hand side, we first observe that
    \begin{equation}\label{eq:a12a1plus}
    \begin{aligned}
        \|(\alpha_1 \oplus \alpha_2)_+\|_{L^2(\Omega)}^2
        & \geq \int_{\tilde{\Omega}^+}(\alpha_1\oplus \alpha_2)^2\d\lambda \\
        & \geq  \int_{\tilde{\Omega}^+} |\alpha_1|^2\d\lambda 
        - 2 \int_{\tilde{\Omega}^+} |\alpha_1|\, |\alpha_2|\,\d\lambda\\
       & \geq  \int_{\tilde{\Omega}^+} |\alpha_1|^2\d\lambda 
       - 2\,\|\alpha_1\|_{L^1(\Omega_1)} \,\|\alpha_2\|_{L^1(\Omega_2)} \\
        & \geq \|\alpha_1\|_{L^2(\tilde{\Omega}^+)}^2  - 2 \, M^2.
    \end{aligned}    
    \end{equation}
    The left hand side in the above equation is estimated by means of \eqref{eq:KantProbL2OptSystem-a} and 
    Lemma~\ref{lem:piest} as follows:
    \begin{equation}\label{eq:a12plusest}
    \begin{aligned}
        \|(\alpha_1 \oplus \alpha_2)_+\|_{L^2(\Omega)}^2
        & \leq 2\left(\int_\Omega(\alpha_1 \oplus \alpha_2 - c)_+^2\d\lambda 
        +\int_\Omega c^2 \d\lambda\right) \\
        & = 2\big(\gamma^2\|\pi_\gamma\|_{L^2(\Omega)}^2 +\|c\|_{L^2(\Omega)}^2\big) \\
        & \leq C \big( \|\mu_1\|_{L^2(\Omega_1)}^2 \,\|\mu_2\|_{L^2(\Omega_2)}^2 + \|c\|_{L^2(\Omega)}^2\big).
    \end{aligned}
    \end{equation}
    Inserting \eqref{eq:a12a1plus} and \eqref{eq:a12plusest} into \eqref{eq:a1plusest} and employing the definition 
    of $M$ as the bound from Lemma~\ref{lem:L1bound} then yields
    \begin{equation}\label{eq:a1plusbound}
        \|(\alpha_1)_+\|_{L^2(\Omega_1)} \leq 
        C \big( \|\mu_1\|_{L^2(\Omega_1)} \|\mu_2\|_{L^2(\Omega_2)}  + \|c\|_{L^2(\Omega)} +1 \big)^2,
    \end{equation}
    with a positive constant $C$ depending on $\gamma$, $m$, and $\delta$. 
    Interchaninging the roles of $\alpha_1$ and $\alpha_2$ implies the same bound for $(\alpha_2)_+$.
    
    Ad (ii): To show the bound for the negative part, we argue similarly to the second part of the proof of 
    \cite[Theorem~2.11]{LMM21}.
    Given $r \in \R$, we consider the set (defined up to sets of zero Lebesgue measure)
	\begin{equation*}
        \hat{\Omega}_2^r \coloneqq 
        \{x_2\in\Omega_2 \colon (\alpha_2)_+(x_2) > r +\cbound\}\subset\Omega_2.
    \end{equation*}
    Recall that $\cbound>-\infty$ is defined to be a lower bound of the cost function $c$, see \eqref{eq:databound}. For any $r > -\cbound$, the Lebesgue measure of this subset can be estimated by
    \begin{equation*}
        |\hat{\Omega}_2^r| 
        \leq \frac{1}{r+\cbound}\int_{\hat{\Omega}_2^r}(\alpha_2)_+\d\lambda_2 
        \leq \frac{1}{r+\cbound}\|\alpha_2\|_{L^1(\Omega_2)} \leq \frac{M}{r +\cbound},
    \end{equation*}
    where $M$ again denotes the bound from Lemma~\ref{lem:L1bound}.
    For all $r > -\cbound$, we thus obtain
    \begin{align*}
        \int_{\Omega_2}(-r+\alpha_2-\cbound)_+\,\d\lambda_2
        & \leq \int_{\Omega_2}(-(r+\cbound)+(\alpha_2)_+)_+\,\d\lambda_2 \\
        & = \int_{\hat{\Omega}_2^r}\big(-(r+\cbound)+(\alpha_2)_+\big) \,\d\lambda_2 \\
        & \leq |\hat{\Omega}_2^r|^\frac12 \,\|(\alpha_2)_+\|_{L^2(\Omega_2)} 
        \leq \Big(\frac{M}{r+\cbound}\Big)^\frac12 K,
    \end{align*}
    where $K$ is short for the bound for $\|(\alpha_2)_+\|_{L^2(\Omega_2)}$ from step (i), i.e., 
    the right-hand side in \eqref{eq:a1plusbound}. Therefore, if we set
    \begin{equation}\label{eq:defN0}
        R \coloneqq \frac{M \,K^2}{\gamma^2 \, \delta^2} - \cbound + 1 > -\cbound,
    \end{equation}                
    then 
    \begin{equation*}
        \int_{\Omega_2}(-R+\alpha_2-\cbound)_+\,\d\lambda_2 < \gamma\,\delta.
    \end{equation*}
    Now assume that $\alpha_1 \leq - R$ $\lambda_1$-a.e.\ on a set $E\subset \Omega_1$ of positive 
    Lebesgue measure. Then
    \begin{equation*}
        \int_{\Omega_2} (\alpha_1 \oplus \alpha_2 - c)_+ \,\d\lambda_2
        \leq \int_{\Omega_2} (-R + \alpha_2 - \cbound)_+ \,\d\lambda_2
        < \gamma\,\delta \leq \gamma\,\mu_1
        \quad \text{$\lambda_1$-a.e.\ in } E,
    \end{equation*}        
    which contradicts \eqref{eq:KantProbL2OptSystem-b}. Hence the definition of $R$ in \eqref{eq:defN0} along with the definition of $M$ and $K$ being the bounds from Lemma~\ref{lem:L1bound} and \eqref{eq:a1plusbound} 
    yields the existence of a constant $C>0$ such that
    \begin{equation*}
        (\alpha_1)_- \leq R \leq 
        C \big( \|\mu_1\|_{L^2(\Omega_1)}\|\mu_2\|_{L^2(\Omega_2)}  + \|c\|_{L^2(\Omega)} +1 \big)^{6} 
        \quad \text{$\lambda_1$-a.e.\ in }\Omega_1,
    \end{equation*}
    i.e., we actually even obtain an $L^\infty$-bound for the negative part. 
    Note that, since $R$ depends on $\gamma$, $\delta$, and $\cbound$, the same holds for the above constant.
    Again, the estimate for $(\alpha_2)_-$ follows by means of reversed roles. 
    Combining the results for the positive and negative part finally proves the lemma.
\end{proof}

With the above boundedness results we are now in a position to prove the (strong) continuity 
of $\SS_\gamma$. For this purpose, let us define the following sets:
\begin{align}
    \HH(\Omega) & \coloneqq L^2(\Omega) \times L^2(\Omega_1) \times L^2(\Omega_2),\\
    \CC_{\cbound}(\Omega)
    & \coloneqq \{ c\in L^2(\Omega) \colon c \geq \cbound
    \text{ $\lambda$-a.e.\ in } \Omega\}, \\
    \MM_\delta^m(\Omega) & \coloneqq \big\{ (\mu_1, \mu_2) \in L^2(\Omega_1) \times L^2(\Omega_2) \colon 
    \begin{aligned}[t]
        & \|\mu_1\|_{L^1(\Omega_1)} = \|\mu_2\|_{L^1(\Omega_2)} = m, \\[-0.5ex]
        & \;\mu_i \geq \delta \text{ $\lambda_i$-a.e.\ in } \Omega_i, \; i=1,2 \big\}.
    \end{aligned}\label{eq:defMMdelta}
\end{align}

\begin{proposition}\label{prop:hoelder}
    Let  $\gamma, \delta, m > 0$ and $\cbound > -\infty$ be given. 
    Then the solution operator of the regularized Kantorovich problem \eqref{eq:KPgam} is 
    H\"older continuous with exponent ${}^1\!/\!{}_2$ on bounded sets in the following sense:
    For all $M > 0$, there exists a constant $L > 0$, depending on $M$, $\gamma$, $\delta$, $m$, and $\cbound$, 
    such that 
    \begin{equation*}
        \| \SS_\gamma(c_\mu, \mu_1, \mu_2) - \SS_\gamma(c_\nu, \nu_1, \nu_2)\|_{L^2(\Omega)}
        \leq L\, \|(c_\mu, \mu_1, \mu_2) - (c_\nu, \nu_1, \nu_2)\|_{\HH(\Omega)}^{{}^1\!/\!{}_2}
    \end{equation*}
    for all $(c_\mu, \mu_1, \mu_2), (c_\nu, \nu_1, \nu_2) \in \CC_{\cbound}(\Omega) \times \MM_\delta^m(\Omega)$ 
    with $$\|(c_\mu, \mu_1, \mu_2)\|_{\HH(\Omega)}, \|(c_\nu, \nu_1, \nu_2)\|_{\HH(\Omega)} \leq M.$$
\end{proposition}

\begin{proof}
    The result is more or less a straight forward consequence of the previous boundedness results.
    Let $(c_\mu, \mu_1, \mu_2), (c_\nu, \nu_1, \nu_2) \in (\CC_{\cbound}(\Omega) \times \MM_\delta^m(\Omega)) 
    \cap \overline{B_{\HH(\Omega)}(0,M)}$ be given.
    We set $\pi_\mu \coloneqq S_\gamma(c_\mu,\mu_1,\mu_2)$ and $\pi_\nu \coloneqq  S_\gamma(c_\nu,\nu_1,\nu_2)$
    and denote the associated zero-mean dual solutions by $(\alpha_1^\mu, \alpha_2^\mu), (\alpha_1^\nu, \alpha_2^\nu) \in L^2(\Omega_1) \times L^2(\Omega_2)$. The equality constraints in \eqref{eq:KPgam} imply
    \begin{equation}\label{eq:KantProbSoluDiff}
        \int_{\Omega_2} (\pi_\mu-\pi_\nu)\,\d{\lambda_2}=\mu_1-\nu_1, \quad
        \int_{\Omega_1} (\pi_\mu-\pi_\nu)\,\d{\lambda_1}=\mu_2-\nu_2.
    \end{equation}
    Testing the first and second equation in \eqref{eq:KantProbSoluDiff} with 
    $\alpha_1^\mu-\alpha_1^\nu$ and $\alpha_2^\mu-\alpha_2^\nu$, respectively, 
    integrating and then adding the resulting equations, we arrive at
    \begin{equation*}
    \begin{aligned}
        & \int_{\Omega} (\pi_\mu-\pi_\nu)
        \Big(\frac1\gamma(\alpha_1^\mu\oplus\alpha_2^\mu-c_\mu)
        - \frac1\gamma(\alpha_1^\nu\oplus\alpha_2^\nu-c_\nu)\Big)\d\lambda \\
        & \quad = \frac1\gamma
        \Big(\int_{\Omega_1}(\mu_1-\nu_1)(\alpha_1^\mu-\alpha_1^\nu)\d{\lambda_1}
        +\int_{\Omega_2}(\mu_2-\nu_2)(\alpha_2^\mu-\alpha_2^\nu)\d{\lambda_2} \\
        & \qquad\qquad - \int_\Omega(\pi_\mu-\pi_\nu)(c_\mu-c_\nu)\d\lambda \Big) .
    \end{aligned}
    \end{equation*}
    Using \eqref{eq:KantProbL2OptSystem-a}, the inequality $(a_+-b_+)^2\leq(a_+-b_+)(a-b)$ for all $a,b\in\R$,
    and Young's inequality, we arrive at 
    \begin{equation*}
        \|\pi_\mu - \pi_\nu\|_{L^2(\Omega)}^2
        \leq C\Big( \sum_{i=1}^2 \|\alpha_i^\mu - \alpha_i^\nu\|_{L^2(\Omega_i)} \|\mu_i -  \nu_i\|_{L^2(\Omega_i)} 
        + \|c_\mu - c_\nu\|_{L^2(\Omega)}^2 \Big).
    \end{equation*}        
    By Lemma~\ref{lem:L2bound}, there holds
    \begin{equation*}
        \|\alpha_i^\mu - \alpha_i^\nu\|_{L^2(\Omega_i)}
        \leq \|\alpha_i^\mu\|_{L^2(\Omega_i)} + \|\alpha_i^\nu\|_{L^2(\Omega_i)}
        \leq C (M+1)^{6}
    \end{equation*}
    with a constant $C$ depending on $\gamma$, $\delta$, $m$, and $\cbound$. 
    Inserting this in the above inequality then gives the result.
\end{proof}

Some words concerning the above results are in order. First, due to their non-uniqueness, 
it is clear that one cannot expect an analogous result for the dual variables, even if we restrict to the 
zero-mean dual solutions. 
Moreover, an inspection of the foregoing analysis reveals that the constant $L$ from Proposition~\ref{prop:hoelder} 
tends to infinity, if $\gamma$ and/or $\delta$ converge to zero or if $\cbound$ approaches $-\infty$.
This is somewhat to be expected, since one looses regularity as $\gamma$ tends to zero and 
there is in general no existence of an optimal transport plan, if there is no lower bound for the cost.

\subsection{Existence of Optimal Solutions}\label{sec:globopt}

Based on the H\"older continuity result, we are now in the position to establish the existence of 
solutions to \eqref{eq:BKgam}. For this purpose, we return to the notation from the beginning of this section, 
i.e. in particular, given $c\in C(\Omega)$ and $\mu_1\in \Measure(\Omega_1)$ with $\mu_1 \geq 0$ and
$\|\mu_1\|_{\Measure(\Omega_1)} = 1$, we set
$\pi_\gamma \coloneqq  
\EE_\delta^* \, \SS_\gamma( \EE_\delta\,c, \TT_1^\delta(\mu_1), \TT_2^\delta(\marg)) \in \Measure(\Omega)$. 

\begin{theorem}\label{thm:BKgamexist}
    Let $\gamma > 0$ and $\delta > 0$ be fixed.
    There exists at least one globally optimal solution to the regularized bilevel Kantorovich problem \eqref{eq:BKgam}.
\end{theorem}

\begin{proof}
    Based on Proposition~\ref{prop:hoelder}, we can apply the direct method of the calculus of variations.
    First, we note that, for all $\mu_1 \in \Measure(\Omega_1)$, $\mu_2 \in \Measure(\Omega_2)$ with 
    $\|\mu_1\|_{\Measure(\Omega_1)} = \|\mu_2\|_{\Measure(\Omega_2)}$ and 
    $\mu_i \geq 0$, $i=1,2$, there holds $\varphi_i^\delta \ast \mu_i \geq 0$ $\lambda_i$-a.e.\ in 
    $\Omega_i^\delta$ and thus
    \begin{equation}
    \begin{aligned}
        \|\TT_1^\delta(\mu_1) \|_{L^1(\Omega_1^\delta)}
        & = \|\varphi_1^\delta\|_{L^1(\R^{d_1})} \|\mu_1\|_{\Measure(\Omega_1)} + \delta \\
        & = \|\varphi_2^\delta\|_{L^1(\R^{d_2})} \|\mu_2\|_{\Measure(\Omega_2)} + \delta 
        = \|\TT_2^\delta(\mu_2) \|_{L^1(\Omega_2^\delta)} \eqqcolon m.
    \end{aligned}
    \end{equation}
    Hence,
    $(\TT_1^\delta(\mu_1), \TT_2^\delta(\marg)) \in \MM_0^m(\Omega_\delta)$ so that $\SS_\gamma( \EE_\delta\,c, \TT_1^\delta(\mu_1), \TT_2^\delta(\marg))$ is well defined 
    for every $c\in C(\Omega)$. Consequently, the feasible set of \eqref{eq:BKgam} is nonempty
    and thus there is an infimal sequence $\{(\pi_\gamma^n, \mu_1^n, c_n)\}_{n\in\N} 
    \subset \Measure(\Omega) \times \Measure(\Omega_1) \times W^{1,p}(\Omega)$, i.e., 
    \begin{equation}
        \JJ_\gamma(\pi_\gamma^n, \mu_1^n, c_n) \to \inf\eqref{eq:BKgam}  \in \R \cup \{-\infty\} \quad
        \text{as } n\to \infty, 
    \end{equation}
    where $\inf\eqref{eq:BKgam}$ denotes the minimal value of \eqref{eq:BKgam}. 
    By the constraint in \eqref{eq:BKgam}, the sequence $\{\mu_1^n\}_{n\in\N}$ is bounded in $\Measure(\Omega_1)$. 
    Moreover, due to the constraints in \eqref{eq:KPgam},
    \begin{equation*}
        \|\pi_\gamma^n\|_{\Measure(\Omega)}
        \leq \|\SS_\gamma( \EE_\delta\,c_n, \TT_1^\delta(\mu_1^n), \TT_2^\delta(\marg))\|_{L^1(\Omega_\delta)}   
         = \int_{\Omega_2^\delta} \varphi_2^\delta \ast \marg \,\d\lambda_2 + \delta
         = 1 + \delta. 
    \end{equation*}
    Therefore, $\{\mu_1^n\}$ and $\{\pi_\gamma^n\}$ are both bounded in $\Measure(\Omega_1)$ 
    and $\Measure(\Omega)$, respectively, and thus, the presupposed weak-$\ast$ lower semicontinuity of $\JJ$ implies that 
    \begin{equation*}
        \inf_{n\in \N} \JJ(\pi_\gamma^n, \mu_1^n) \eqqcolon j > -\infty.
    \end{equation*}
    Consequently, for $n\in \N$ large enough,  
    \begin{equation*}
        \|c_n - \cost\|_{W^{1,p}(\Omega)}^p 
        \leq p\gamma \big(\max\{\inf \eqref{eq:BKgam} +1,0\} - j\big)
    \end{equation*}
    such that $\{c_n\}$ is bounded in $W^{1,p}(\Omega)$.
    Therefore, there is a subsequence, denoted for simplicity by the same symbol, such that 
    \begin{equation}\label{eq:weakconv}
        (\pi_\gamma^n, \mu_1^n) \weak^* (\bar\pi, \bar \mu_1) 
        \quad \text{in } \Measure(\Omega) \times \Measure(\Omega_1), \quad 
        c_n \weak \bar c \quad \text{in } W^{1,p}(\Omega).
    \end{equation}
    The set $\{\mu_1\in \Measure(\Omega_1) \colon \mu_1 \geq 0\}$ is easily seen to be weakly-$\ast$ closed.
    Hence $\bar\mu_1 \geq 0$ and thus 
    \begin{equation*}
        1 = \|\mu_1^n\|_{\Measure(\Omega_1)} 
        = \int_{\Omega_1} \d\mu_1^n \to \int_{\Omega_1}\d \bar\mu_1 = \|\bar\mu_1\|_{\Measure(\Omega_1)}.
    \end{equation*} 
    Therefore, the weak limit $\bar\mu_1$ satisfies the constraints in \eqref{eq:BKgam}.    
    Moreover, since $p>d$, the embedding $W^{1,p}(\Omega) \embed C(\Omega)$ is compact, 
    and thus $c_n \to \bar c$ uniformly in $\Omega$. Since $\Omega$ is compact, there is a constant 
    $\cbound > -\infty$ such that $\inf_{x\in \Omega} c_n(c) \geq \cbound$ for all $n\in \N$
    as well as $\inf_{x\in \Omega}\bar c(x) \geq \cbound$.
    Hence, there holds $\EE_\delta \bar c, \EE_\delta c_n \in \CC_{\cbound}(\Omega_\delta)$ for all $n\in \N$ 
    and the uniform convergence implies 
    \begin{equation*}
        \EE_\delta c_n \to \EE_\delta \bar c \quad \text{in } L^2(\Omega_\delta).
    \end{equation*}
    Furthermore, the complete continuity of the convolution yields
    \begin{equation*}
        \TT_1^\delta(\mu_1^n) \to \TT_1^\delta(\bar\mu_1) 
        \quad \text{in } L^2(\Omega_1^\delta).
    \end{equation*}
    Since $\mu_1^n, \bar\mu_1 \geq 0$ and $\marg \geq 0$, we have 
    $\varphi_1^\delta\ast \mu_1^n, \varphi_1^\delta \ast \bar\mu_1 \geq 0$ $\lambda_1$-a.e.\ in 
    $\Omega_1^\delta$ and $\varphi_2^\delta \ast \marg \geq 0$ $\lambda_2$-a.e.\ in $\Omega_2^\delta$ and 
    therefore thanks to the constant shift in \eqref{eq:defTT}
    \begin{equation*}
        (\TT_1^\delta(\mu_1^n), \TT_2^\delta(\marg)) \in \MM_{\delta^\prime}^m(\Omega_\delta) \quad \forall\, n\in \N, 
        \quad 
        (\TT_1^\delta(\bar\mu_1), \TT_2^\delta(\marg)) \in \MM_{\delta^\prime}^m(\Omega_\delta),         
    \end{equation*}
    where $\delta^\prime \coloneqq  \max\{|\Omega_1^\gamma|, |\Omega_2^\gamma|\}^{-1}\,\delta$.
    Therefore, we are allowed to apply the H\"older continuity result from Proposition~\ref{prop:hoelder}, which gives
    \begin{equation*}
        \SS_\gamma( \EE_\delta\,c_n, \TT_1^\delta(\mu_1^n), \TT_2^\delta(\marg))
        \to \SS_\gamma( \EE_\delta\,\bar c, \TT_1^\delta(\bar\mu_1), \TT_2^\delta(\marg))
        \quad \text{in } L^2(\Omega_\delta).
    \end{equation*}
    The continuity of the restriction operator $\EE_\delta^*$ and the first convergence in \eqref{eq:weakconv} 
    then imply 
    \begin{equation*}
        \bar\pi = \EE_\delta^* \,\SS_\gamma( \EE_\delta\,\bar c, \TT_1^\delta(\bar\mu_1), \TT_2^\delta(\marg))
    \end{equation*}
    such that the weak limit $(\bar\pi, \bar\mu_1, \bar c)$ is feasible for \eqref{eq:BKgam}. 
    The mapping $W^{1,p}(\Omega) \ni c \mapsto \|c - \cost\|_{W^{1,p}(\Omega)}^p$ is continuous and convex and therefore weakly lower semicontinuous. Additionally, $\JJ$ was assumed to be weakly-$\ast$ lower semicontinuous. Altogether, the convergence in 
    \eqref{eq:weakconv} yields
    \begin{equation*}
        \JJ_\gamma(\bar\pi, \bar\mu_1, \bar c)
        \leq \liminf_{n\to\infty} \JJ_\gamma(\pi_n, \mu_1^n, c_n) = \inf\eqref{eq:BKgam},
    \end{equation*}
    ensuring the optimality of the weak limit.
\end{proof}


\bibliographystyle{plain}
\bibliography{refs}

\begin{thebibliography}{10}

\bibitem{AF03}
Robert~A. Adams and John J.~F. Fournier.
\newblock {\em Sobolev spaces}, volume 140 of {\em Pure and Applied Mathematics
  (Amsterdam)}.
\newblock Elsevier/Academic Press, Amsterdam, second edition, 2003.

\bibitem{AG13}
Luigi Ambrosio and Nicola Gigli.
\newblock A user’s guide to optimal transport.
\newblock In {\em Modelling and optimisation of flows on networks}, pages
  1--155. Springer, 2013.

\bibitem{Bar93}
V.~Barbu.
\newblock {\em Analysis and Control of nonlinear infinite dimensional systems}.
\newblock Academic Press, New York, 1993.

\bibitem{CCMW18}
Constantin Christof, Christian Clason, Christian Meyer, and Stephan Walther.
\newblock Optimal control of a non-smooth semilinear elliptic equation.
\newblock {\em Math. Control Relat. Fields}, 8(1):247--276, 2018.

\bibitem{CdlRM20}
Constantin Christof, Juan~Carlos De~los Reyes, and Christian Meyer.
\newblock A nonsmooth trust-region method for locally {L}ipschitz functions
  with application to optimization problems constrained by variational
  inequalities.
\newblock {\em SIAM J. Optim.}, 30(3):2163--2196, 2020.

\bibitem{CLM21}
Christian Clason, Dirk~A. Lorenz, Hinrich Mahler, and Benedikt Wirth.
\newblock Entropic regularization of continuous optimal transport problems.
\newblock {\em J. Math. Anal. Appl.}, 494(1):Paper No. 124432, 22, 2021.

\bibitem{Cut13}
Marco Cuturi.
\newblock Sinkhorn distances: {L}ightspeed computation of optimal transport.
\newblock In {\em Advances in neural information processing systems}, pages
  2292--2300, 2013.

\bibitem{Gri85}
Pierre Grisvard.
\newblock {\em Elliptic Problems in Nonsmooth Domains}.
\newblock Classics in Applied Mathematics. SIAM, Philadelphia, 1985.

\bibitem{HU19}
Lukas Hertlein and Michael Ulbrich.
\newblock An inexact bundle algorithm for nonconvex nonsmooth minimization in
  {H}ilbert space.
\newblock {\em SIAM J. Control Optim.}, 57(5):3137--3165, 2019.

\bibitem{HMW12}
Roland Herzog, Christian Meyer, and Gerd Wachsmuth.
\newblock {C}-stationarity for optimal control of static plasticity with linear
  kinematic hardening.
\newblock {\em SIAM Journal on Control and Optimization}, 50(5):3052--3082,
  2012.

\bibitem{HMM22}
Sebastian Hillbrecht, Paul Manns, and Christian Meyer.
\newblock Bilevel optimization of the {K}antorovich problem and its quadratic
  regularization, part {II}: Convergence results in infinite dimensions.
\newblock In preparation, 2022.

\bibitem{HM22}
Sebastian Hillbrecht and Christian Meyer.
\newblock Bilevel optimization of the {K}antorovich problem and its quadratic
  regularization, part {III}: Reverse approximation in finite dimensions.
\newblock In preparation, 2022.

\bibitem{HK09}
M.~Hinterm{\"u}ller and I.~Kopacka.
\newblock Mathematical programs with complementarity constraints in function
  space: C-and strong stationarity and a path-following algorithm.
\newblock {\em SIAM Journal on Optimization}, 20(2):868--902, 2009.

\bibitem{HKS13}
Tim Hoheisel, Christian Kanzow, and Alexandra Schwartz.
\newblock Theoretical and numerical comparison of relaxation methods for
  mathematical programs with complementarity constraints.
\newblock {\em Math. Program.}, 137(1-2, Ser. A):257--288, 2013.

\bibitem{Kan42}
Leonid~V. Kantorovi{\v{c}}.
\newblock On the translocation of masses.
\newblock {\em C. R. (Doklady) Acad. Sci. URSS (N.S.)}, 37:199--201, 1942.

\bibitem{KS15}
Christian Kanzow and Alexandra Schwartz.
\newblock The price of inexactness: convergence properties of relaxation
  methods for mathematical programs with complementarity constraints revisited.
\newblock {\em Math. Oper. Res.}, 40(2):253--275, 2015.

\bibitem{ML21}
Dirk Lorenz and Hinrich Mahler.
\newblock Orlicz space regularization of continuous optimal transport problems.
\newblock arXiv:2004.11574, 2021.

\bibitem{LMM21}
Dirk~A. Lorenz, Paul Manns, and Christian Meyer.
\newblock Quadratically regularized optimal transport.
\newblock {\em Appl. Math. Optim.}, 83(3):1919--1949, 2021.

\bibitem{LPR96}
Zhi-Quan Luo, Jong-Shi Pang, and Daniel Ralph.
\newblock {\em Mathematical programs with equilibrium constraints}.
\newblock Cambridge University Press, Cambridge, 1996.

\bibitem{MW21}
Christian Meyer and Stephan Walther.
\newblock Optimal control of perfect plasticity {P}art {II}: {D}isplacement
  tracking.
\newblock {\em SIAM J. Control Optim.}, 59(4):2498--2523, 2021.

\bibitem{Mig76}
F.~Mignot.
\newblock Contr\^ole dans les in\'equations variationelles elliptiques.
\newblock {\em Journal of Functional Analysis}, 22(2):130--185, 1976.

\bibitem{San15}
Filippo Santambrogio.
\newblock {\em Optimal transport for applied mathematicians}, volume~87 of {\em
  Progress in Nonlinear Differential Equations and their Applications}.
\newblock Birkh\"{a}user/Springer, Cham, 2015.
\newblock Calculus of variations, PDEs, and modeling.

\bibitem{SS00}
Holger Scheel and Stefan Scholtes.
\newblock Mathematical programs with complementarity constraints: stationarity,
  optimality, and sensitivity.
\newblock {\em Math. Oper. Res.}, 25(1):1--22, 2000.

\bibitem{SW13}
A.~Schiela and D.~Wachsmuth.
\newblock Convergence analysis of smoothing methods for optimal control of
  stationary variational inequalities with control constraints.
\newblock {\em ESAIM: M2AN}, 47:771--787, 2013.

\bibitem{Vil03}
C\'{e}dric Villani.
\newblock {\em Topics in optimal transportation}, volume~58 of {\em Graduate
  Studies in Mathematics}.
\newblock American Mathematical Society, Providence, RI, 2003.

\bibitem{Vil09}
C\'{e}dric Villani.
\newblock {\em Optimal transport. Old and new}, volume 338 of {\em Grundlehren
  der Mathematischen Wissenschaften [Fundamental Principles of Mathematical
  Sciences]}.
\newblock Springer-Verlag, Berlin, 2009.

\bibitem{Wac16}
Gerd Wachsmuth.
\newblock Towards {M}-stationarity for optimal control of the obstacle problem
  with control constraints.
\newblock {\em SIAM Journal on Control and Optimization}, 54(2):964--986, 2016.

\bibitem{Wac19}
Gerd Wachsmuth.
\newblock A guided tour of polyhedric sets.
\newblock {\em Journal of Convex Analysis}, 26(1):153--188, 2019.

\end{thebibliography}

\end{document}